\documentclass[12pt,reqno]{amsart}
\usepackage[top=1in, bottom=1in, left=1in, right=1in]{geometry}

\usepackage[usenames,dvipsnames]{color}
\usepackage{times}
\usepackage{tikz}
\usepackage{ifthen}
\usepackage{tikz-3dplot}
\usepackage{pgfplots}
\usepackage{pifont}
\usepackage{amssymb}
\usepackage{mathrsfs}
\usepackage{bbm}
\usepackage{mathabx}

\numberwithin{equation}{section}

\tikzset{
    plane max x/.initial=2,
    plane max y/.initial=2,
    plane max z/.initial=2
}

\newtheorem{theorem}{Theorem}[section]
\newtheorem{proposition}[theorem]{Proposition}
\newtheorem{lemma}[theorem]{Lemma}

\theoremstyle{definition}
\newtheorem{definition}[theorem]{Definition}
\newtheorem{example}[theorem]{Example}
\newtheorem{remark}[theorem]{Remark}
\newtheorem{notation}[theorem]{Notation}

\newtheorem{convention}[theorem]{Convention}

\newcommand\<{\langle}
\newcommand\CC{{\mathbb C}}
\newcommand\NN{{\mathbb N}}
\newcommand\RR{{\mathbb R}}
\newcommand\ZZ{{\mathbb Z}}

\newcommand\sG{{\mathscr G}}

\newcommand\QQ{{\mathbb Q}}
\newcommand\kk{{\mathbbm{k}}}

\newcommand{\tv}{\tilde{v}}

\newcommand\del{\partial}

\newcommand{\Ext}{\operatorname{Ext}}

\newcommand\Horn{{\rm Horn}}

\newcommand\minus{\smallsetminus}

\DeclareMathOperator{\qdeg}{qdeg}

\renewcommand\>{\rangle}

\newcommand\numberthis{\addtocounter{equation}{1}\tag{\theequation}}

\begin{document}
\mbox{}
\title{Primary components of codimension two lattice basis ideals}

\author{Zekiye Sahin Eser}
\address{Department of Mathematics \\
Texas A\&M University \\ College Station, TX 77843.}
\email{sahin@math.tamu.edu}

\author{Laura Felicia Matusevich}
\address{Department of Mathematics \\
Texas A\&M University \\ College Station, TX 77843.}
\email{laura@math.tamu.edu}

\thanks{
The authors were partially supported by NSF Grant DMS 1001763.
}

\subjclass[2010]{Primary: 13F99, 52B20 ; Secondary: 33C70, 20M25}

\begin{abstract}
We provide explicit combinatorial descriptions of the primary
components of codimension two lattice basis ideals. 
As an application, we compute the set of parameters for which a
bivariate Horn system of hypergeometric differential equations is
holonomic. 
\end{abstract}
\maketitle

\section{Introduction}
\label{sec:introduction}

Let $\kk$ be a field.
A binomial is an element of a polynomial ring $\kk[x]=\kk[x_{1},\dots,x_{n}]$
with at most two terms; a \emph{binomial ideal} is one whose generators can be
chosen binomial. The zero sets of binomial ideals are unions of toric
varieties, which makes binomial ideals important in algebraic geometry, and is one
reason that combinatorial methods are very effective for
studying them. Beyond their intrinsic interest, binomial ideals
arise naturally in various contexts, such as combinatorial game theory, algebraic
statistics and dynamics of mass action kinetics, see \cite{Mil11a,
  alg-stat, shiu-sturmfels}, 
for references and details. These ideals are also very
important in the study of hypergeometric 
differential equations, as is shown in~\cite{DMM,BMW}.

Binomial ideals are highly structured. Assuming that the field $\kk$ is
algebraically closed, Eisenbud and Sturmfels 
\cite{ES 98} have shown that the associated primes and primary
components of a binomial ideal can be chosen binomial. In addition,
when $\kk$ is of characteristic zero, a combinatorial description
of the primary components of a binomial ideal was provided by
Dickenstein, Matusevich and Miller in \cite{Dick et al 08}. This
description involves graphs whose vertices are the lattice points in certain
cones in $\RR^n$, and are thus difficult to visualize when the number of 
variables is high. 

If $I$ is a binomial ideal, there exists a (multi)grading of the
polynomial ring that makes $I$ a homogeneous ideal. In fact, it is often the case
that a binomial ideal is given together with a specified grading.
Depending on their behavior with respect to a given grading, the
associated primes and primary components of a binomial ideal
are called \emph{toral} or \emph{Andean} (see
Definition~\ref{def:toral Andean}).
In general, the toral primary components of a binomial ideal are more
easily understood combinatorially than the Andean ones, as their
graphs can actually be drawn in much lower dimension than the number
of variables. 

In this article, we study the primary decomposition of
\emph{codimension two lattice basis ideals,} which have a
prescribed grading (see Definition~\ref{def:lattice basis},
Convention~\ref{conv:AB} and the paragraph before
Definition~\ref{def:qdeg - Andean}).
In this case, it is known that the graphs controlling their toral
primary components have vertices in $\NN^2$, regardless of the
dimension of the ambient polynomial ring. Our goal here is to
understand the combinatorics of the Andean primary components of a
codimension two lattice basis ideal; this is achieved in
Theorem~\ref{thm:component}. To prove this result, we construct an
infinite family of 
graphs with vertices in $\NN^2$, that nevertheless contain enough information to
produce the needed primary components. 

As was mentioned before, binomial ideals, and in particular, their
primary decompositions, play a key role in the study of hypergeometric
functions. Using our new understanding of the Andean components of a
codimension two lattice basis ideal, we are able to compute the
\emph{Andean arrangement} (Definition~\ref{def:qdeg - Andean}) of a system of
hypergeometric differential equations in two variables, which consists
of all the parameters for which such a system possesses a finite
dimensional solution space (Theorem~\ref{thm:Andean arrangement}).


\subsection*{Acknowledgements}
%
We are grateful to Christine Berkesch Zamaere, Thomas Kahle, Ezra
Miller, and Christopher O'Neill 
for interesting conversations while we were working on this project,
and for comments on a previous version of this article.

\section{Graphs associated to matrices}
\label{sec:graphs}

Throughout this article, $\NN$ denotes the monoid $\{0,1,2,\dots\}$.

One of the main aims in \cite{Dick et al 08} is to give a
combinatorial description of the primary components of binomial
ideals. This description involves graphs whose vertices belong to
submonoids of $\ZZ^n$, and more precisely, the connected
components of those graphs (see, for instance,
Theorem~\ref{thm:component for monomial associated prime}). In this
section, we study graphs arising from $2\times 2$ 
integer matrices, whose vertices are elements of $\NN^2$. These graphs
are simpler than those introduced in \cite{Dick et al 08}, but it
turns out that they are sufficient to control the primary components
of the special kind of binomial ideal we are interested in. 

We collect our results on graphs here, since they require no algebraic
preliminaries, and may be of independent combinatorial interest.
In Sections~\ref{sec:three variables} and~\ref{sec:hypergeometric}, we
use these graphs to compute the 
primary components of codimension two lattice basis ideals.

\begin{definition}
\label{def:graphs for matrices}
Let $Q$ be a subset of $\ZZ^n$ and let $M$ be an $n\times m$
integer matrix. We define a graph $G_Q(M)$ whose vertices are the
elements of $Q$, and where two vertices $u,v \in Q$ are connected by an
edge if and only if $u-v$ or $v-u$ is a column of $M$. A connected
component of $G_Q(M)$ is  called \emph{infinite}, if it contains
infinitely many vertices; otherwise it is called \emph{finite}. A
\emph{finite} (or \emph{infinite}) vertex of $G_Q(M)$ is one that
belongs to a finite (or infinite) connected component. If $Q=\NN^n$,
we omit $Q$ from the notation, and write $G(M)$ instead of $G_{\NN^n}(M)$.
\end{definition}

In this article, we consider graphs $G_Q(M)$ where $Q$ is a submonoid
of $\ZZ^n$ such as $\NN^n$ or $\NN^k\times \ZZ^{n-k}$
(for instance, in Proposition~\ref{prop : inf ver of I and inf ver of
  J.}), or a subset of  
$\ZZ^n$, such as $\{ u \in \NN^n \mid \lambda_1 u_1 + \cdots + \lambda_n u_n =
\lambda_0\}$, for fixed given $\lambda_0,\dots,\lambda_n \in \QQ$
(Lemma~\ref{lemma:straightening graphs}), or $\{ u \in \NN^2
\mid u_1 \leq \ell\}$, for fixed given $\ell \in \NN$
(Proposition~\ref{prop : xyz-matrix.}).

We are interested in the connected components of $G_Q(M)$, and in
particular, in determining whether these connected components are
infinite or finite.

Our first result concerns the connected components of $G(M)$, where
$M$ is a $2\times 2$ nonsingular matrix whose rows lie in
non adjacent (also called \emph{opposite}) open quadrants of $\ZZ^2$, and
$Q=\NN^2$ is omitted from the notation.

\begin{proposition}[Lemma~6.5 in~\cite{DMS}]
\label{propo:2x2 toral graphs}
Let $M=[\mu_{ij}]_{i,j \in \{1,2\}} \in \ZZ^{2\times 2}$ of rank two, and assume that
$\mu_{11}, \mu_{12} >0$ and  $\mu_{21},\mu_{22} <0$. Set
\[
\mathscr{R} = 
\left\{ \begin{array}{ll}
\{ u \in \NN^2 \mid u_1 < \mu_{12}, u_2<-\mu_{21} \} \quad & 
\text{if } |\mu_{11}\mu_{22}|>|\mu_{12}\mu_{21}|, \\
\{ u \in \NN^2 \mid u_1 < \mu_{11}, u_2<-\mu_{22} \} \quad & 
\text{if } |\mu_{11}\mu_{22}|<|\mu_{12}\mu_{21}| .
\end{array}
\right.
\]
Every finite connected component of $G(M)$ contains exactly one vertex in
$\mathscr{R}$. In particular, the number of finite connected components of
$G(M)$ is the cardinality of $\mathscr{R}$, which is
$\min(|\mu_{11}\mu_{22}|,|\mu_{12}\mu_{21}|)$. 
\end{proposition}

\begin{example}

Let $M = \begin{bmatrix} 
\,\, 1 & \,\,3 \\
-2 & -4
 
\end{bmatrix}$. $G(M)$ has min$(|-4|, |-6|) = 4$ finite connected
components, which are shown in Figure~\ref{fig:toral}.

\begin{figure}[h]
\begin{center}
\begin{tikzpicture}[thick,scale=0.6]
\draw [dotted, gray] (0,0) grid (6,6);
\draw [->] (0,0) -- (0,6);
\draw [thick, ->] (0,0) -- (6,0);
\draw [thick] (1,0) -- (0,2);
\draw [thick] (2,0) -- (1,2);
\draw [thick] (1,2) -- (0,4);
\draw [thick] (3,0) -- (0,4);
\draw [thick] (3,0) -- (2,2);
\draw [thick] (2,2) -- (1,4);
\draw [thick] (1,1) -- (0,3);
\draw [thick] (1,4) -- (0,6);
\draw [thick] (3,2) -- (0,6);
\draw [dotted] (2.5,3.5) -- (3,4);
\node [above] at (0,6) {$z$};
\node [right] at (6,0) {$w$};
\foreach \Point in{((2,0), (1,2), (0,4), (3,0), (2,2), (1,4), (0,6), (3,2)}{
  \node at \Point {\textbullet};
}
\foreach \Point in{(0,0), (1,1), (0,3), (1,0), (0,2), (0,1)}{
  \node at \Point {\ding{83}};
}
\end{tikzpicture}

\end{center}
\caption{The graph of $M$.}
\label{fig:toral}
\end{figure}
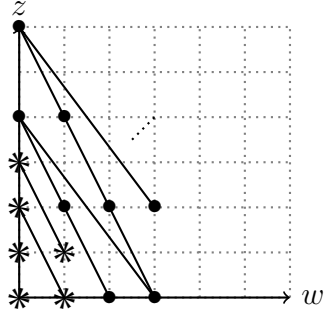
\end{example}

If $M\in \ZZ^{2\times 2}$ has all positive entries, then
$G(M)$ has no finite connected components. However, in this case, it
is not $G(M)$ that carries the algebraic information we need later on, but a
family of subgraphs of $G(M)$, as follows.

\begin{notation}
\label{notation:bands}
Let $M \in \ZZ^{2\times 2}$ of rank two, all of whose entries are
positive. For $\ell \in \NN$ let $Q_\ell = \{ u \in \NN^2 \mid u_1 \leq
\ell\}$. We denote $G_\ell(M) \coloneq G_{Q_\ell}(M)$, and call these graphs
the \emph{band graphs} of $M$. Note that $G_\ell(M)$ is the induced
subgraph of $G(M)$ whose vertices lie in $Q_\ell$, and consequently if 
$\ell \leq \ell'$, then $G_{\ell}(M)$ is a subgraph of $G_{\ell'}(M)$.
\end{notation}

In order to understand the connected components of the band subgraphs
of $G(M)$, we first analyze a special case.

\begin{proposition}
\label{prop : xyz-matrix.}
Consider a rank two integer matrix 
\[ 
M =
\begin{bmatrix}
r & s \\
a & b
\end{bmatrix}
\]
such that $r \geq s >0$, $0 < a \leq b$, and $\gcd(r,s) = 1$.
If $\ell < r+s -1$, then every connected component of $G_{\ell}(M)$ is
finite. If $ \ell \geq r+s-1$, then for every $w_0 \in
\{0,1,\dots,\ell\}$ there exists $z_0 \in \NN$ such that $(w_0,z_0)$
belongs to an infinite component of $G_{\ell}(M)$.
\end{proposition}

Before proving this result, we introduce some auxiliary notions, and
provide some examples.

Given $M$ as above, the graphs $G_\ell(M)$ have two types of edges:
those parallel to the first column of $M$ are called the
\emph{$r$-edges} of $G_\ell(M)$, and those parallel to the
second column of $M$ are called the \emph{$s$-edges} of $G_\ell(M)$. If $r = s$, we could refer to these edges as $a$-edges and $b$-edges.

\begin{definition} 
\label{def:turns}
Let $M$ be as in Theorem~\ref{prop : xyz-matrix.}, and consider $\NN^2$
with coordinates $w,z$.
A vertex of $G_{\ell}(M)$ is called a \emph{turn} if it is
adjacent to both an $r$-edge and an $s$-edge of $G_\ell(M)$. 
A turn $(w_0,z_0)$ is called a \emph{left turn} if there is a vertex 
adjacent to $(w_0,z_0)$ whose $w$-coordinate is smaller than $w_0$.
Turns that are not left turns are called \emph{right turns}. 
Intuitively, when we walk along a connected component of
$G_{\ell}(M)$ in the direction that increases $z$,
we turn left at a left turn, and right at a right turn.
\end{definition}

\begin{example} 
\label{example:2diml slice graphs}
Let  

\[
  M = \begin{bmatrix} 
\,\, 7 & \,\,4 \\
 \,\,1 & \,\,1 
\end{bmatrix}
 \]
The band graphs $G_{4}(M)$ and $G_{7}(M)$
are illustrated in Figure~\ref{fig:finite slices}.

\begin{figure}[h]
\begin{center}
\begin{minipage}{.4\textwidth}
\begin{tikzpicture}[thick,scale=0.8]
\draw [dotted, gray] (0,0) grid (4,3);
\draw [<->] (0,3) -- (0,0) -- (4,0);
\draw [thick] (0,0) -- (4,1);
\draw [thick] (0,1) -- (4,2);
\draw[thick, dashed] (2,2) -- (2,3);
\node [above] at (0,3) {$z$};
\node [right] at (4,0) {$w$};
\node [above] at (2,-1) {$G_{4}(M)$};
\foreach \Point in{(0,0), (4,1), (4,2), (0,1)}{
  \node at \Point {\textbullet};
}
\end{tikzpicture}
\end{minipage}%
\begin{minipage}{.4\textwidth}
\begin{tikzpicture}[thick,scale=0.8]
\draw [dotted, gray] (0,0) grid (7,3);
\draw [<->] (0,3) -- (0,0) -- (7,0);
\draw [thick] (0,0) -- (7,1);
\draw [thick] (0,0) -- (4,1);
\draw [thick] (3,0) -- (7,1);
\draw [thick] (3,1) -- (7,2);
\draw [thick] (0,1) -- (7,2);
\draw [thick] (0,1) -- (4,2);
\draw[thick, dashed] (4,2) -- (4,3);
\draw [->, dashed] (7,2) -- (8,1.5);
\draw [->, dashed] (7,1) -- (8,1.5);
\node [above] at (0,3) {$z$};
\node [right] at (7,0) {$w$};
\node [above] at (3,-1) {$G_{7}(M)$};
\node [right] at (8,1.5) {left turn};
\foreach \Point in{(0,0), (4,1), (3,1), (7,1), (3,0), (7,2), (0,1), (4,2)}{
  \node at \Point {\textbullet};
}
\end{tikzpicture}
\end{minipage}%
\end{center}
\caption{Examples of band graphs.}
\label{fig:finite slices}
\end{figure}
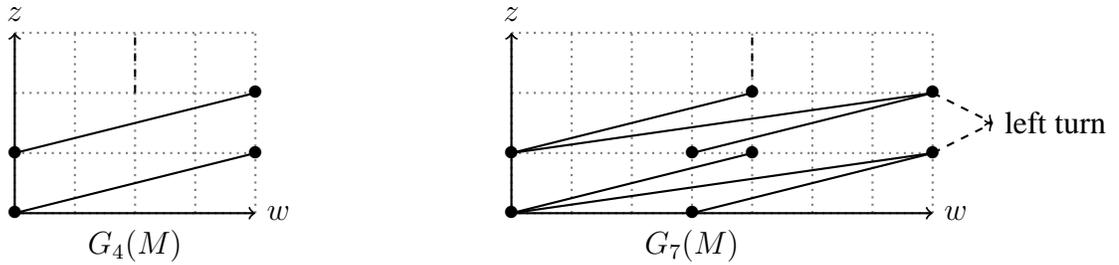

All of the connected components of $G_{4}(M)$ and
$G_{7}(M)$ are finite. The minimum $\ell \in \NN$ such that
$G_{\ell}(M)$ has an infinite connected component is $\ell = 10$
(see Figure~\ref{fig:infinite slice}).

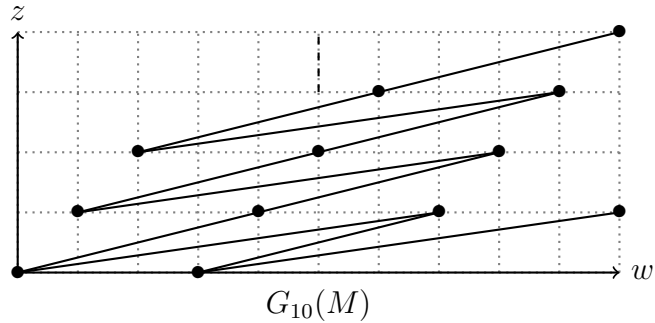
\begin{figure}[h]
\begin{center}
\begin{tikzpicture}[thick,scale=0.8]
\draw [dotted, gray] (0,0) grid (10,4);
\draw [<->] (0,4) -- (0,0) -- (10,0);
\draw [thick] (0,0) -- (7,1);
\draw [thick] (8,2) -- (4,1);
\draw [thick] (0,0) -- (4,1);
\draw [thick] (3,0) -- (7,1);
\draw [thick] (1,1) -- (5,2);
\draw [thick] (1,1) -- (8,2);
\draw [thick] (3,0) -- (10,1);
\draw [thick] (2,2) -- (9,3);
\draw [thick] (6,3) -- (2,2);
\draw [thick] (6,3) -- (10,4);
\draw [thick] (5,2) -- (9,3);
\draw[thick, dashed] (5,3) -- (5,4);
\node [above] at (0,4) {$z$};
\node [right] at (10,0) {$w$};
\node [above] at (5,-1) {$G_{10}(M)$};
\foreach \Point in{(0,0), (4,1), (7,1), (8,2), (3,0), (10,1), (1,1), (2,2), (9,3), (6,3), (10,4), (5,2)}{
  \node at \Point {\textbullet};
}
\end{tikzpicture}
\caption{A band graph with an infinite component.}
\label{fig:infinite slice}
\end{center}
\end{figure}
\end{example}

\begin{proof}[Proof of Proposition~\ref{prop : xyz-matrix.}]
  Write $r = s q_{1} + q_{2}$ where $0\leq q_{2} < s$. We claim that any
  connected component 
  of $G_{r}(M)$ contains at most $2q_{1}+2$ vertices (implying that
  $G_{r}(M)$, and therefore $G_\ell(M)$ for $\ell \leq r$, has no
  infinite connected components). A connected 
  component of $G_{r}(M)$ can only contain one $r$-edge since the
  $w$-coordinates of vertices in $G_{r}(M)$ are bounded by
  $r$. Thus, we can have at most two turns in such a connected
  component. We can connect at most $q_{1}$-many $s$-edges at each
  turn. Including the turns, the number of vertices in a connected
  component of $G_r(M)$ is at most equal to $2q_1+2$. 

  We observe that a modification of the argument above shows that a connected
  component of $G_\ell(M)$ is infinite if and only if it contains
  infinitely many left turns.

  Now consider $G_{r + t}(M)$ where $0 \leq t < s$. We show that
  $G_{r + t}(M)$ has an infinite connected component if and
  only if $t = s - 1$. Note that not all components of $G_{r + t}(M)$
  have left turns, for instance, the vertex $(r+t,0)$ is itself a
  connected component, which therefore has no turns. In what follows,
  we study how many left turns a connected component can have.

  The ordering $\succ$ on the elements of $\NN^2$ defined by
  $(w,z) \succ (w',z')$ if $z > z'$, or $z = z'$ and $w' > w$, induces
  a total ordering on the set of left turns of a given component of $G_{r +
    t}(M)$. 

   Let $C$ be a connected component of $G_{r+t}(M)$, and suppose that
   $(w_0,z_0)$ is a left turn in $C$. We wish to produce the next left
   turn of $C$ according to $\succ$, if it exists. Since
   $(w_0,z_0)$ is a left turn in $C$, we have $(w_0-r,z_0-a) \in
   C$. This is a right turn, 
   because $G_{r + t}(M)$ cannot contain two adjacent $r$-edges, as the
   $w$-coordinates of the vertices of $G_{r + t}(M)$ are bounded by $r
   + t$, and $t<s$. We attach $s$-edges to $(w_0-r,z_0-a)$, to produce
   a vertex $(w_0-r,z_0-a)+ (qs,qb) \in C$, where $q>0$ is as large as
   possible. The integer $q$ is produced by writing $r+t-(w_0-r) =
   qs+[r+t-(w_0-r) \mod s]$, where $[\alpha \mod \beta]$ denotes the
   remainder of $\alpha$ upon division by $\beta$, for $\alpha, \beta
   \in \ZZ$, $\alpha>0$.
 
   If $(w_0-r,z_0-a)+ (qs,qb)$ is
   coordinatewise greater than or equal to $(r,a)$, then 
   $(w_0-r,z_0-a)+ (qs,qb)$ is a left turn of
   $C$ which is greater according to $\succ$ than $(w_0,z_0)$. Now, 
   $z_0-a \geq 0$ and $b \geq a$ imply that $z_0 -a + qb \geq
   a$. Therefore, in order for $(w_0-r,z_0-a)+ (qs,qb)$ to be a left
   turn, we need $r \leq w_0-r+qs = r+t - [r+t-(w_0-r) \mod s]$, or
   equivalently, $t \geq [2r+t - w_0 \mod s]$.

   Replacing $(w_0,z_0)$ by $(w_0-r,z_0-a)+ (qs,qb)$, we see that the
   condition needed for the existence of a left turn which is
   greater according to $\succ$ than  $(w_0-r,z_0-a)+ (qs,qb)$ is
   $t \geq [2r+t - (w_0-r +qs) \mod s] = [3r+t - w_0 \mod s]$.
   
   Continuing in this manner, the existence of infinitely many left
   turns in $C$ is equivalent to requiring $t \geq [\ell r +t -w_0
   \mod s]$ for all $\ell > 0$. However, since 
   $\gcd(r,s)=1$, there exists $\ell> 0$ 
   such that $[(\ell r + t - w_0 \mod s] = s-1$. Therefore, if
   $t<s-1$, $C$ has finitely many left turns, and is finite, and if
   $t=s-1$, $C$ has infinitely many left turns, and is infinite.

   If $t<s-1$, a component of $G_{r+t}(M)$ either has no left turns or
   finitely many left turns, which shows that $G_{r+t}(M)$ has no
   infinite components.

   Let $t=s-1$ and $w_0 \in \{0,\dots, r+s-1\}$. If $w_0\geq r$, then for
   large enough $z_0$, $(w_0,z_0)$ is a vertex of both an $r$- and an
   $s$-edge whose other vertex has lower $z$-coordinate, and is
   therefore a left turn in its connected component, 
   which is thus infinite. If $w_0 < r$, we can choose $z_0$
   sufficiently large such that 
   attaching as many $s$-edges to $(w_0,z_0)$ as possible yields a
   left turn, which implies that the component of $(w_0,z_0)$ is infinite.

   Finally, if $\ell > r+s-1$, for each $0\leq t \leq \ell -(r+s-1)$,
   $G_\ell(M)$ contains as a subgraph the image of $G_{r+s-1}(M)$
   under the translation $(w,z) \mapsto (w+t,z)$. This implies that
   for each $w_0 \in \{0,\dots,\ell\}$, there is $z_0 >0$ such that
   $(w_0,z_0)$ is an infinite vertex of $G_\ell(M)$.
\end{proof}

In the previous statement, we assumed that the entries in the top
column of $M$ were relatively prime. We now remove that assumption.

\begin{theorem}
\label{prop : dilation.} 
Consider a rank two matrix 
\[ 
M =
\begin{bmatrix}
r & s \\
a & b
\end{bmatrix}
\]
where $r, s, a, b$ are positive integers, $r \geq s$, $a \leq b$
and $\gcd(r,s) = d \geq 1$.
The minimal $\ell \in \NN$ such that $G_{\ell}(M)$ has an infinite connected
component is $\ell = r + s - d$. If $0 \leq t < d$ and $w_0 \in
\{0,\dots,r+s-d+t\}$, there exists $z_0$ such that $(w_0,z_0)$ is an
infinite vertex of $G_{r+s-d+ t}(M)$ if and only if $w_0$ is
divisible by $d$.
If $\ell > r+s$, for each $w_0 \in
\{0,\dots,\ell\}$, there exists $z_0$ such that $(w_0,z_0)$ is an
infinite vertex of $G_{\ell}(M)$.
\end{theorem}

\begin{example}
\label{ex:bands}
 Let $ M = \begin{bmatrix} 
 2 & 6 \\
1 & 2
 
\end{bmatrix}$. When $\ell = 6$, the band graph $G_{6}(M)$ has an
infinite connected component. However, the vertices $(w, z)$ where $w$ is odd
are finite vertices for all $z$; see Figure~\ref{fig:inf-finite-bands}.

\begin{figure}[h]
\begin{center}
\begin{tikzpicture}[thick,scale=0.6]
\draw [dotted, gray] (0,0) grid (6,6);
\draw [->] (0,0) -- (0,6);
\draw [thick, ->] (0,0) -- (6,0);
\draw [thick] (4,1) -- (2,0);
\draw [thick] (4,1) -- (6,2);
\draw [thick] (0,0) -- (6,2);
\draw [thick] (0,0) -- (2,1);
\draw [thick] (2,1) -- (4,2);
\draw [thick] (4,2) -- (6,3);
\draw [thick] (6,3) -- (0,1);
\draw [thick] (0,1) -- (2,2);
\draw [thick] (2,2) -- (4,3);
\draw [dashed] (3,0) -- (5,1);
\draw [dashed] (1,0) -- (3,1);
\draw [dashed] (3,1) -- (5,2);
\draw [dashed] (1,1) -- (3,2);
\draw [dashed] (3,2) -- (5,3);
\draw [dashed] (1,2) -- (3,3);
\draw [dashed] (3,3) -- (5,4);
\draw [dotted] (3.5,3.5) -- (3.5,4.5);
\node [above] at (0,6) {z};
\node [right] at (6,0) {w};
\foreach \Point in{(4,1), (2,0), (6,2), (0,0), (2,1), (4,2), (6,3), (0,1), (2,2), (4,3)}{
  \node at \Point {\textbullet};
}
\foreach \Point in{(1,0), (1,1), (1,2), (3,0), (3,1), (3,2), (5,0), (5,1), (5,2), (5,4), (3,3), (5,3)}{
  \node at \Point {\ding{83}};
}
\end{tikzpicture}
\caption{The band graph $G_{6}(M)$. \label{fig:inf-finite-bands}}
\end{center}
\end{figure}
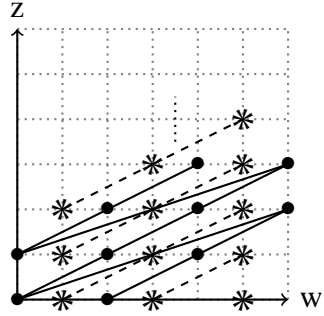

\end{example}

\begin{proof} 
  Let $\hat{M}$ be the (integer) matrix obtained from $M$ by dividing
  $r$ and $s$ by $d$, so that Proposition~\ref{prop : xyz-matrix.}
  applies to the band graphs of $\hat{M}$. 

  Let $\ell \in \NN$ and set $\hat{\ell} = \left \lfloor{\ell/d}
  \right \rfloor $, the integer part of $\ell/d$. We show that $G_\ell(M)$ is a
  disjoint union of graphs isomorphic to $G_{\hat{\ell}}(\hat{M})$ or
  $G_{\hat{\ell}-1}(\hat{M})$. 

  Let $(w_0,z_0) \in \NN^2$ such that $w_0 \leq \ell$, so that
  $(w_0,z_0)$ is a vertex of $G_\ell(M)$. Write
  $w_0 = \hat{w}_0 d + t_0$ where $t_0$ is an integer with $0 \leq t_0
  < d$. 

  If $t_0 = 0$, then $(w_0,z_0)$ belongs to the image of
  the map $\varphi_{\ell,0} : G_{\hat{\ell}}(M) \to G_\ell(M)$
  defined on vertices by $(w,z) \mapsto (dw,z)$. Since $r$ and $s$ are
  divisible by $d$, any vertex in the 
  connected component of $G_\ell(M)$ that contains $(w_0,z_0)$ also
  has its $d$-coordinate divisible by $d$. This implies that the
  connected component of $(w_0,z_0)$ in $G_\ell(M)$ is the image under
  $\varphi_{\ell,0}$ of the connected component of $(w_0/d=\hat{w}_0,z_0)$ in
  $G_{\hat{\ell}}(\hat{M})$.

  If $t_0 > 0$, consider the map  $\varphi_{\ell,t_0} :
  G_{\hat{\ell}-1}(M) \to G_\ell(M)$ defined on vertices by $(w,z) \mapsto
  (dw+t_0,z)$. Since $r$ and $s$ are divisible by $d$, the
  $w$-coordinates of all the vertices of $G_\ell(M)$ connected to
  $(w_0,z_0)$ are congruent to $t_0$ modulo $d$. This implies that the
  connected component of $(w_0,z_0)$ in $G_\ell(M)$ is the image under
  $\varphi_{\ell,t_0}$ of the connected component of $(\hat{w}_0,z_0)$ in
  $G_{\hat{\ell}-1}(\hat{M})$.

  Note that the images of the maps $\varphi_{\ell,i}$ have no common
  vertices, and their union is $G_\ell(M)$. Now use
  Proposition~\ref{prop : xyz-matrix.} to obtain the desired conclusions. 
\end{proof}

We now explain how to construct graphs arising from binomial
ideals. Lemma~\ref{lemma:chain} relates these
graphs and those associated to matrices in the case of lattice basis ideals.

\begin{definition} 
\label{def:graphs}
Let $P$ be a monoid. A binomial ideal $I$ in the monoid ring $\kk[P]$
defines a graph $\sG_P(I)$ whose vertices are the elements of $P$ and
whose edges are pairs $(u,v) \in P\times P$ such that $x^{u}-\rho
x^{v} \in I$ for some $\rho \in \kk \minus \{0\}$. A connected component of 
$\sG_P(I)$ is said to be \emph{infinite} if it consists of infinitely many
vertices; otherwise it is called \emph{finite}. A vertex of $\sG_P(I)$ is
called an \emph{infinite vertex} if it 
belongs to an infinite connected component, otherwise it is called a
\emph{finite vertex}. If $P=\NN^n$, we omit $P$ from the notation, and
write $\sG(I)$ instead of $\sG_P(I)$.
\end{definition} 

Note that if $I \subset \kk[P]$ is a binomial ideal, any connected
component of $\sG_P(I)$ is a complete graph.

We are ready to introduce our main objects of study. 

\begin{definition}
\label{def:lattice basis}
If $\mu \in \ZZ^n$, define $\mu_+,\mu_- \in \NN^n$ via $(\mu_+)_i =
\max(\mu_i,0)$ and $(\mu_-)_i = \max(-\mu_i,0)$, so that $\mu=\mu_+-\mu_-$.
Let $M$ be $n\times m$ integer matrix. We define a binomial ideal
associated to $M$ as follows:
\[
I(M) = \< x^{\mu_+}-x^{\mu_-} \mid \mu \text{ is a column of } M\> \subseteq
\kk[x_1,\dots,x_n] = \kk[\NN^n].
\]
If $M$ has rank $m$, then $I(M)$ is called a \emph{lattice basis ideal}.
\end{definition}

\begin{lemma}
\label{lemma:chain}
Let $M$ be an $n\times m$ integer matrix of rank $m$, and $I(M)$ its
corresponding lattice basis ideal as in
Definition~\ref{def:lattice basis}. Let $\tau\subseteq
\{1,\dots,n\}$ and let $P = \NN^{\tau} \times \ZZ^{\bar{\tau}}$.
Then $u,v \in P$ are connected
in $\sG_P(\kk[P] \cdot I(M))$ if and only if they are connected in $G_P(M)$.
\end{lemma}

\begin{proof}
Assume that $u,v \in P$ are connected in $G_P(M)$. We show that
$x^u-x^v \in \kk[P]\cdot I(M)$ by 
induction on the length of the path connecting $u$ to $v$. If this
path has length one, then $u$ and $v$ are connected by an edge of
$G_P(M)$, meaning that $u-v$ or $v-u$, say $u-v$, equals a column $\mu$ of
$M$. Then $u - v = \mu = \mu_+ - \mu_-$, so that $v-\mu_- = u-\mu_+ \eqcolon 
\nu$. Since $v-\mu_- + \mu_+ = u \in P$ and for all $i$,
$(\mu_+)_i$ and $(\mu_-)_i$ are not simultaneously nonzero, we see
that $\nu \in P$. But then $x^u - x^v = x^{\nu}(x^{\mu_+} -
x^{\mu_-}) \in \kk[P] \cdot I(M)$, as we wished.

Now assume that $u$ and $v$ are connected in $G_P(M)$ by a path of
length $\ell>1$. This means that there are vertices
$u=\nu^{(0)},\nu^{(1)},\dots,\nu^{(\ell)} 
= v$ of $G_P(M)$ such that $(\nu^{(i)},\nu^{(i+1)})$ is an edge of
$G_P(M)$ for $i=0,\dots, \ell$.
By inductive hypothesis, since $\nu^{(1)}$ and $v$ are connected in
$G_P(M)$ by a path of length $\ell-1$, we have $x^{\nu^{(1)}}-x^v \in
\kk[P]\cdot I(M)$. But we also know $x^u-x^{\nu^{(1)}} \in \kk[P]\cdot
I(M)$. We conclude that
$x^u-x^v \in \kk[P]\cdot I(M)$, and therefore $u$ and $v$ are connected in
$\sG_P(I(M))$.

For the converse, we start by noting that a lattice basis ideal (and
its extension to $\kk[P]$)
contains no monomials. This follows, for instance, from Lemma~7.6
in~\cite{miller-sturmfels}, which implies that the saturation
$(I(M):\<x_1\cdots x_n\>^\infty)\subseteq \kk[x]$ is not the unit ideal.

Since every connected component of $\sG_P(I(M))$ is a complete graph, if
$u$ and $v$ are connected in $\sG_P(I(M))$, then $(u,v)$ is an edge in $\sG_P(I(M))$.
Thus, there exists nonzero $\rho \in \kk$ such that $x^u-\rho x^v \in
\kk[P] \cdot I(M)$, and if $\mu^{(1)}, \dots,\mu^{(m)}$ are the columns of $M$, we
  can write $x^u-\rho x^v = F_1(x) (x^{\mu^{(1)}_+}-x^{\mu^{(1)}_-}) +
  \cdots + F_m(x) (x^{\mu^{(m)}_+}-x^{\mu^{(m)}_-})$ for certain
  $F_1,\dots,F_m \in \kk[P]$. We can represent this expression as a
  subgraph $K$ of $G_P(M)$: for every term $\lambda x^{\nu}$ in $F_i$,
  $K$ contains the edge $(\nu + \mu^{(i)}_+,\nu + \mu^{(i)}_-)$
  and its corresponding vertices. We label this edge by the
  coefficient $\lambda$, and we label each vertex by the combination
  of the labels of the edges adjacent to it, with a positive sign if
  we look at the vertex $\nu + \mu^{(i)}_+$ of $(\nu + \mu^{(i)}_+,\nu
  + \mu^{(i)}_-)$, and a negative sign for the vertex $\nu + \mu^{(i)}_-$.
  Thus, the only two vertices with nonzero labels are $u$ and $v$.

  Let $K_u$ be the connected component of
  $K$ containing $u$. We wish to show that $v$ is a vertex in $K_u$,
  as this implies that $u$ and $v$ are connected in $G_P(M)$. 
  But if this is not the case, we can use $K_u$ to form a
  polynomial expression with a summand $\lambda x^{\nu}
  (x^{\mu^{(i)}_+}-x^{(\mu^{(i)}_-}) $ for each 
  edge $(\nu + \mu^{(i)}_+,\nu + \mu^{(i)}_-)$ labeled by $\lambda$ in
  $K_u$, and this expression equals the sum over the vertices in $K_u$
  of the label of each vertex times the corresponding monomial. Since
  the only vertex with a nonzero label in $K_u$ is $u$ (that label is
  $1$), then we obtain an expression for $x^u$ as a combination of the
  generators of $\kk[P]\cdot I(M)$. This contradicts the fact that
  $\kk[P] \cdot I(M)$ contains
  no monomials. 
\end{proof}

With the hypotheses and notation of the previous result, we see that
we can construct $\sG_P(I(M))$ by adding edges to $G_P(M)$ until
each connected component becomes a complete graph. 

Given an arbitrary binomial ideal $I\subseteq \kk[P]$, it is always
possible to construct a 
subgraph of $\sG_P(I)$ using a generating set of $I$, so that the
underlying vertex sets of their connected components are the same (and
therefore, saturating the connected components of this subgraph with
edges yields $\sG_P(I)$). For this purpose, not every generating set of
$I$ contains sufficient information. What is needed is a generating
set of $I$ that contains all the generators of the maximal
monomial ideal in $I$. The statement (and proof) of this
generalization are more technical, but follow along the same lines as above.
We point out that special cases of this result can be found in the
literature, for example Lemmas~1 and~2 in \cite{MM}.

\section{Primary decomposition of binomial ideals}
\label{sec:background}

In this section we review important facts about
the primary decomposition of binomial ideals, and especially of
lattice basis ideals (Definition~\ref{def:lattice basis}).

Recall that $\NN=\{0,1,2,\dots\}$. We work in the
polynomial ring $\kk[x_1,\dots,x_n]=\kk[\NN^n]$ where $\kk$ is an algebraically
closed field. Unless otherwise specified, $\kk$ is of characteristic zero
(in Section~\ref{sec:hypergeometric}, we use $\kk=\CC$).

The associated primes of lattice basis ideals were studied by
Ho\c{s}ten and Shapiro in \cite{HS00}; they show that the minimal
primes of such an ideal are determined by the sign patterns of
the entries of the corresponding matrix. In this article, we study
lattice basis ideals arising from $n\times 2$ integer matrices, known
as \emph{codimension two lattice basis ideals}.

\begin{convention}
\label{conv:AB}
From now on, $B=[b_{ij}]$ is an $n\times 2$ integer matrix of rank
$2$. The rows of $B$ are denoted by $b_1,\dots,b_n$, and its columns
by $B_1,B_2$. Fix an integer $(n-2)\times n$ matrix $A$ such that
$AB=0$, and whose columns span $\ZZ^{n-2}$ as a lattice. 
\end{convention}

Since $B$ has rank two, the ideal $I(B)$ from
Definition~\ref{def:lattice basis} is a complete intersection.
Therefore all of its associated primes are minimal.
By Corollary~2.1 in~\cite{HS00}, the set of associated
primes of $I(B)$ consists of the associated primes of 
$(I : (\prod^{n}_{i =1}{x_{i}})^{\infty})$ and the monomial primes $\langle
x_{i}, x_{j} \rangle$ if $b_{i}$ and $b_{j}$ lie in opposite open quadrants
of $\ZZ^{2}$.

All of the associated primes of $(I : (\prod^{n}_{i =1}{x_{i}})^{\infty})$
are isomorphic, by rescaling the variables, to
$I_A=\<x^{v_+}-x^{v_-}\mid v \in \ZZ^n, Av=0\>$, where $A$ is as
in Convention~\ref{conv:AB}. The prime ideal $I_A$ is called the
\emph{toric ideal} associated to $A$. It is shown in \cite{ES 98} that
the primary components of $I(B)$ 
corresponding to these associated primes are the associated primes
themselves (since the characteristic of the underlying field $\kk$ is
zero). Thus, we now turn our attention to the primary components 
of $I(B)$ arising from monomial associated primes. 

One of the main results in~\cite{ES 98} is that any binomial prime ideal
in $\kk[x]$ is of the form $\kk[x] \cdot J +\<x_i \mid i \in \sigma\>$,
where $\sigma \subseteq \{1,\dots,n\}$ and $J \subset
\kk[x_j\mid j \notin \sigma]$ is isomorphic to a toric ideal by rescaling
the variables. In characteristic zero, the primary component of a
binomial ideal $I$ corresponding to such an associated prime is of
the form $([I+J]:[\prod_{j \notin \sigma} x_j]^\infty)+
\mathscr{M}$, where $\mathscr{M}$ is a monomial ideal generated by
elements of $\kk[x_i \mid i \in \sigma]$. A key idea from~\cite{Dick et al 08}
is that the graphs from Definition~\ref{def:graphs} can be used to
determine the monomials in the primary components of a binomial
ideal. 
A specific result in this vein is Theorem~\ref{thm:component
  for monomial associated prime} below, which  describes the primary
component of a binomial ideal corresponding to a monomial associated
prime.

We first set up some notation.
Given $\sigma\subseteq \{1,\dots, n\}$, we set $\bar{\sigma} =
\{1,\dots,n\} \minus \sigma$. Denote 
$\NN^{\sigma} = \{ u \in \NN^n \mid u_i = 0 \text{ for } i\notin
\sigma\}$, and  
$\ZZ^{\bar{\sigma}} = \{ u \in \ZZ^n \mid u_i = 0 \text{ for } i \in
\sigma\}$. We consider $P=\NN^\sigma \times \ZZ^{\bar{\sigma}}$ as
  a submonoid of $\ZZ^n$. The corresponding monoid ring is
  $\kk[P]=\kk[\NN^\sigma\times \ZZ^{\bar{\sigma}}] = \kk[x_{j}^{\pm}\mid
  j\notin \sigma ][x_{i}\mid i \in \sigma]$.

\begin{theorem}[Theorem~2.15, \cite{Dick et al 08}]
\label{thm:component for monomial associated prime} 
Let $\kk$ be an algebraically closed field (of any characteristic) and
$I\subset \kk[x]$ a binomial ideal. Let $\sigma \subseteq \{1,\dots,
n\}$, and set $P= \NN^\sigma\times \ZZ^{\bar{\sigma}}$.
If $\<x_i \mid i \in \sigma\>$ is a minimal prime of $I$, its
corresponding primary component is  
\begin{equation}
\label{eqn:component description}
\bigg(I:(\prod_{j\notin \sigma} x_j)^\infty\bigg) + \<x^u \mid u\in \NN^n \text{ is an
  infinite vertex of } \sG_P(\kk[P]\cdot I)\>.
\end{equation}
Moreover, the only monomials in these primary components are those of
the form $x^u$ such that $u \in \NN^n$ is an infinite vertex of
$\sG_P(\kk[P]\cdot I)$. 
\end{theorem}

\begin{remark}
\label{rmk:generated in sigma vars}
Note that the monomial ideal in~\eqref{eqn:component description} is
generated by monomials $x^u$ where $u_j = 0$ if $j \notin
\sigma$. Indeed, if $u \in \NN^n$ is an infinite vertex of
$\sG_P(\kk[P]\cdot I)$, so is $u - \hat{u}$, where $\hat{u}_i=0$ if $i\in
\sigma$ and $\hat{u}_j=u_j$ if $j \notin \sigma$. This is because
monomials in the variables $x_j$ for $j \notin \sigma$ are units in
$\kk[P]$. 
\end{remark}

The following useful criterion to identify the infinite components of 
$\sG_P(\kk[P] \cdot I)$ is a special case of Lemma~2.10 in~\cite{Dick et al 08}.

\begin{lemma}
\label{lemma : inf comp.} 
  Let $I \subseteq k[x_{1},\dots,x_{n}]$ be a binomial ideal, $\sigma
  \subseteq \{1,\dots,n\}$, 
  and $P = \NN^{\sigma}\times \ZZ^{\bar{\sigma}}$.
  A connected 
  component of $\sG_P(\kk[P] \cdot I)$ is infinite if and only if it contains two
  distinct vertices $u, v \in P$ such that $u-v \in P$.
\end{lemma}

The graphs $\sG_P(\kk[P]\cdot I)$ used in Theorem~\ref{thm:component
  for monomial associated prime} 
have vertices in $P=\NN^{\sigma}\times \ZZ^{\bar{\sigma}}$. When $n$ is
large, we cannot hope to compute the finite 
(or infinite) connected components of $\sG_P(\kk[P]\cdot I)$ by simply drawing
the graph. In certain situations, other graphs, whose ambient lattice
is lower dimensional, contain the same information as
$\sG+P(\kk[P]\cdot I)$. This happens, for instance, when $\<x_i \mid i \in
\sigma\>$ behaves well with respect to a given grading, as follows.

The matrix $A$ in Convention~\ref{conv:AB} can be used to define a
$\ZZ^{n-2}$-grading of $\kk[x]$, where $\deg(x_i)$ is defined to be the $i$th
column of $A$. The ideal $I(B)$ is homogeneous with respect to this
\emph{$A$-grading} (as are therefore its associated primes and primary
components). The associated primes (and primary components) of an
$A$-graded ideal are classified according to their $A$-graded behavior.

\begin{definition}
\label{def:toral Andean}
Let $I \subseteq \kk[x]$ be an $A$-graded binomial ideal, and 
$\mathfrak{p}$ an associated prime of $I$. If the $A$-graded Hilbert
function of $\kk[x]/\mathfrak{p}$ is bounded, then $\mathfrak{p}$ and
the corresponding primary component of $I$, are called
\emph{toral}. Otherwise, they are called \emph{Andean}.
\end{definition}

\begin{example}
\label{ex:codim2 toral Andean}
Among the associated primes of the codimension two lattice basis
ideal $I(B)$, the only Andean ones are the monomial primes $\langle
x_i,x_j\rangle$ such that the corresponding rows of $B$, $b_i$ and
$b_j$, are linearly dependent (and in opposite open quadrants of $\ZZ^2$). 

To see this, first observe that the $A$-graded Hilbert function of
$\kk[x]/I_A$ takes only the values zero and one (and the same holds for the other
non-monomial associated primes of $I(B)$). On the other hand, whether the
Hilbert function of $\kk[x]/\<x_i,x_j\>$ is bounded depends on the rank
of the submatrix of $A$ indexed by $\{1,\dots,n\}\minus \{i,j\}$:
boundedness is equivalent to this submatrix having full rank, and this
happens exactly when $b_i$ and $b_j$ are linearly independent.
\end{example}

The toral components of an $A$-graded binomial ideal are more
accessible combinatorially than the Andean ones in general. The
following result illustrating this fact is a consequence of
Theorem~4.13 from~\cite{Dick et al 08}. 

\begin{theorem}
\label{thm:toral components}
Let $I$ be an $A$-graded binomial ideal in $\kk[x]$, where $k$ is
algebraically closed of characteristic zero, and assume that
$\mathfrak{p}=\<x_i \mid i \in \sigma\>$ is a toral minimal prime of $I$. Define
the binomial ideal
$\bar{I} = I \cdot \kk[x]/\<x_j - 1 \mid j \notin \sigma\> \subset \kk[\NN^{\sigma}]$
by setting $x_j=1$ for $j \notin \sigma$. The $\mathfrak{p}$-primary
component of $I$ is:
\[
\bigg(I:(\prod_{j\notin \sigma} x_j)^\infty\bigg) + \<x^u \mid u\in \NN^{\sigma} \text{ is an
  infinite vertex of } \sG(\bar{I})\>.
\]
\end{theorem}

The main feature of the previous theorem is that 
$\sG(\bar{I})$ has vertices in $\NN^{\sigma}$, and the
cardinality of $\sigma$ can be 
much smaller than $n$. In the case of codimension two lattice basis
ideals, $\NN^\sigma=\NN^2$, regardless of the number of variables
$n$. In fact, if $\<x_i,x_j\>$ is a toral associated prime of $I(B)$,
then $\overline{I(B)} = \< x_i^{|b_{i1}|} - x_j^{|b_{j1}|},
x_i^{|b_{i2}|} - x_j^{|b_{j2}|}\> \subset \kk[x_i,x_j]$, and by
  Lemma~\ref{lemma:chain}, $\sG(\overline{I(B)})$ can be
  constructed from the graph $G(M)$ associated to the $2\times 2$ matrix
  $M$ whose rows are $b_i$ and $b_j$. But we have already
  characterized the connected components of $G(M)$ in
  Proposition~\ref{propo:2x2 toral graphs}, so we can 
  describe the corresponding primary component by applying
  Theorem~\ref{thm:toral components}.

When $\kk$ is of characteristic zero, this yields a very satisfactory
picture of the toral components of a codimension two lattice basis
ideal: the primary components corresponding to non-monomial associated
primes are isomorphic to toric ideals by rescaling the variables; the
primary components corresponding to monomial toral primes are
described by Theorem~\ref{thm:toral components} and
Proposition~\ref{propo:2x2 toral graphs}. 

\begin{example}
Let 
\[
  B = \begin{bmatrix} 
\phantom{-}2 & \phantom{-}4 \\
-4 & -6 \\
\phantom{-}2 & \phantom{-}3 \\
-1 &\phantom{-}3 \\
-1 & -2 \\
\phantom{-}2 & -6 \\
-8 & -12\\
-3 & -6
\end{bmatrix},
\]
so that $I(B) = \langle x_{1}^{2}x_{3}^{2}x_{6}^{2} -
x_{2}^{4}x_{4}x_{5}x_{7}^{8}x_{8}^{3}, x_{1}^{4}x_{3}^{3}x_{4}^{3} -
x_{2}^{6}x_{5}^{2}x_{6}^{6}x_{7}^{12}x_{8}^{6} \rangle$. 

The rows $(2, 4)$ and $(-4, -6)$ are linearly independent and lie in
opposite open quadrants, so we consider
$ M = \begin{bmatrix} 
 \phantom{-}2 & \phantom{-}4 \\
-4 & -6 
\end{bmatrix}$. 
We can compute the monomials of the $\langle x_{1},
x_{2}\rangle$-primary component of $I(B)$ by looking the graph of
$G(M)$; see Figure~\ref{fig:monomials}.

\begin{figure}[h]
\begin{center}

\begin{tikzpicture}[thick,scale=0.6]
\draw [dotted, gray] (0,0) grid (7,9);
\draw [->] (0,0) -- (0,9);
\draw [thick, ->] (0,0) -- (7,0);
\draw [thick] (2,4) -- (0,8);
\draw [thick] (4,2) -- (0,8);
\draw [thick] (4,2) -- (2,6);
\draw [thick] (2,2) -- (0,6);
\draw [thick] (4,0) -- (0,6);
\draw [thick] (4,0) -- (2,4);
\draw [dashed] (0,6) -- (2,6);
\draw [dashed] (2,6) -- (2,2);
\draw [dashed] (2,2) -- (4,2);
\draw [dashed] (4,2) -- (4,0);
\draw [dotted] (4,4) -- (3.5,3.5);
\node [above] at (0,9) {$u_{y}$};
\node [right] at (7,0) {$u_{x}$};
\foreach \Point in{(2,2), (4,0), (0,6)}{
  \node at \Point {\ding{109}};
}
\foreach \Point in{(2,4), (4,0), (0,6), (2,2), (0,8), (4,2), (2,6)}{
  \node at \Point {\textbullet};
}
\end{tikzpicture}
\caption{The graph of $G(M)$ \label{fig:monomials}}
\end{center}
\end{figure}
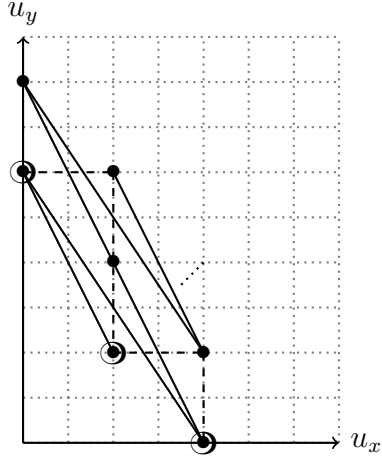

By Theorem~\ref{thm:toral components}, the $\langle x_{1},
x_{2}\rangle$-primary component of $I(B)$ is: 

\[ (I(B) : (\prod_{\ell \neq  1, 2}{x_{\ell}})^{\infty}) + \langle
x_{1}^{4}, x_{2}^{6}, x_{1}^{2}x_{2}^{2}\rangle = \langle x_{1}^{4},
x_{2}^{6}, x_{1}^{2}x_{2}^{2},  x_{2}^{4}x_{4}x_{5}x_{7}^{8}x_{8}^{3}
- x_{1}^{2}x_{3}^{2}x_{6}^{2}\rangle.\] 

\end{example}

\section{Codimension two lattice basis ideals in three variables}
\label{sec:three variables}

In this section, we study the Andean components of codimension two
lattice basis ideals in three variables.

\begin{notation}
\label{conv : matrix.} 
For this section only, and except where otherwise noted, we let $B$ be
$3\times 2$ matrix of full rank $2$ as follows:
\[
B = \begin{bmatrix}
\phantom{-\lambda} r & \phantom{-\lambda} s \\
-\lambda r & -\lambda s \\
\phantom{-\lambda} a & \phantom{-\lambda} b
\end{bmatrix}
\]
where $r, s, a, b \in \ZZ_{>0}$, $a \leq b$, $r \geq s$, 
$\gcd(r,s) = d$, and
$0<\lambda = p/q$ in lowest terms. 
We work in the polynomial ring $\kk[x,y,z]$. 
The lattice basis ideal associated to $B$ is $I(B) = \langle
x^{r}z^{a} - y^{\lambda r}, 
x^{s}z^{b} - y^{\lambda s} \rangle \subseteq \kk[x,y,z]$.
We let $P =\NN^2 \times \ZZ$, and work with $\kk[P] = \kk[z^{\pm}][x,y]$.
\end{notation}

\begin{remark} 
\label{rmk: graphs in hyperplanes}
For a codimension two lattice basis ideal $I(B)$ arising from a $3\times
2$ matrix $B$ as in Notation~\ref{conv : matrix.}, 
a vertex $u=(u_x,u_y,u_z)$ of $\sG(I(B))$ that lies on a hyperplane $u_y =
-\lambda u_x + \lambda \ell$, for $\ell \in \QQ$, can only be connected
to other vertices on that hyperplane. This follows from
Lemma~\ref{lemma:chain} since the columns of $B$ are parallel to the
hyperplane 
$u_y=-\lambda u_x$.
\end{remark}

For $B$ as above, we wish to compute the primary component of
$I(B)$ corresponding to the (Andean) associated prime $\<x,y\>$.
According to Theorem~\ref{thm:component
  for monomial associated prime}, we need to
understand the graph arising from the extension of $I(B)$ to
$\kk[z^{\pm}][x,y] = \kk[P]$. 

The following proposition shows that all the
information we need about the infinite components of
$\sG_P(\kk[P]\cdot I(B))$ is contained in the infinite components 
of $\sG(I(B))$, consequently the graph $G(B)$ can be used to compute the
primary component of $I(B)$ corresponding to the associated prime $\<x,y\>$.

\begin{proposition}
\label{prop : inf ver of I and inf ver of J.} 
  Let $B$ and $P=\NN^2 \times \ZZ$ as in Notation~\ref{conv : matrix.}. Then
\begin{align*}
\{ (u_x,u_y) \in \NN^2 \mid \exists u_z \in \ZZ \text{ such that }
(u_x,u_y,u_z) \text{ is an infinite vertex of } \sG_P(\kk[P] \cdot I(B))\}
= \\
\{ (u_x,u_y) \in \NN^2 \mid \exists u_z \in \NN \text{ such that }
(u_x,u_y,u_z) \text{ is an infinite vertex of } \sG(I(B))\}.
\end{align*}
  Consequently, the primary component of $I(B)$ corresponding to the
  associated prime $\<x,y\>$ is:
\[
(I:z^\infty) + \< x^{u_x} y^{u_y} \mid \exists u_z \in \NN \text{ such that }
(u_x,u_y,u_z) \text{ is an infinite vertex of } G(B)\>.
\]
\end{proposition}

\begin{proof} 
  If $u = (u_{x}, u_{y}, u_{z}) \in \NN^3$ is an infinite vertex of
  $\sG(I(B))$, then it is clear that it is also an infinite vertex of
  $\sG_P(\kk[P]\cdot I(B))$. 

  Let ${u} =
  ({u}_{x},{u}_{y},{u}_{z}) \in P$ be an infinite vertex
  of of $\sG_P(\kk[P]\cdot I(B))$. By Lemma~\ref{lemma : inf comp.} there
  exists $v=(v_x,v_y,v_z), \tv=(\tv_x,\tv_y,\tv_z) \in \NN^2 \times
  \ZZ$ connected to ${u}$ 
  such that $\tv_x \geq v_x$ and ${\tv}_y \geq {v}_y$.
  Since ${u}$ is connected to ${v}$, we can find a nonzero
  $\rho \in k$ such that $x^{u_x}y^{u_y}z^{u_z} - \rho
  x^{v_x}y^{v_y}z^{v_z} \in \kk[P]\cdot I(B)$, and by clearing
  denominators, we can produce $\mu \in \NN$
  such that $z^\mu(x^{u_x}y^{u_y}z^{u_z} - \rho
  x^{v_x}y^{v_y}z^{v_z}) \in I(B)$; in particular, $\mu + u_z$ and
  $\mu+v_z$ are non negative. Thus, the vertices $(u_x,u_y,u_z+\mu),
  (v_x,v_y,v_z+\mu) \in \NN^3$ are connected in $\sG(I(B))$. Enlarging
  $\mu$ as needed, we may assume that 
  $(u_x,u_y,u_z+\mu),
  (v_x,v_y,v_z+\mu), (\tv_x,\tv_y,\tv_z+\mu)$ are coordinatewise non
  negative and connected in $\sG(I(B))$.

  By
  Remark~\ref{rmk: graphs in hyperplanes}, there exists $\ell \in \QQ$ such that
  $v_y=-\lambda v_x + \lambda\ell$ and $\tv_y=-\lambda \tv_x+\lambda
  \ell$, so that $\tv_y-v_y = -\lambda (\tv_x-v_x)$. Since $\lambda>0$ and
  $\tv_x-v_x$, $\tv_y-v_y$ are non negative, we see that $v_x=\tv_x$ and $v_y=\tv_y$.
  
  In conclusion, the vertices $(v_x,v_y,v_z+\mu), (v_x,v_y,\tv_z+\mu)
  \in \NN^3$ are
  connected in $G(I(B))$; since either $(v_x,v_y,v_z+\mu)
  -(v_x,v_y,\tv_z+\mu)$ or $(v_x,v_y,\tv_z+\mu) - (v_x,v_y,v_z+\mu)$
  belongs to $\NN^3$, we see that these vertices belong to an infinite
  component of $\sG(I(B))$ by Lemma~\ref{lemma : inf comp.}. As these
  vertices are connected to $(u_x,u_y,u_z+\mu)$, we conclude that this is
  an infinite vertex of $\sG(I(B))$.

  The last statement now follows from 
  Theorem~\ref{thm:component
  for monomial associated prime}, Remark~\ref{rmk:generated in
    sigma vars} and Lemma~\ref{lemma:chain}.
\end{proof}

For the remainder of this section, we study the connected components
of the graph $G(B)$ for $B$ as in Notation~\ref{conv : matrix.}. 
The following definition is motivated by Remark~\ref{rmk: graphs in hyperplanes}.

\begin{definition}
\label{def : slice graph} 
For a matrix $B$ as in Notation~\ref{conv : matrix.} and $\ell \in
(1/p)\NN$, set $S(\ell) = \NN^n \cap \{ (u_x,u_y,u_z) \mid u_y =
-\lambda u_x + \lambda \ell \}$. The graph $G_{S(\ell)}(B)$
is called the \emph{$\ell$th-slice graph of $B$}. Slice graphs are
illustrated in Figure~\ref{fig:slices}. 
\end{definition}

By Remark~\ref{rmk: graphs in hyperplanes}, $G(B)$ equals 
the (disjoint) union $\bigcup_{\ell \in
  (1/p)\NN} G_{S(\ell)}(B)$ as a graph. 

\begin{figure}[h]
\begin{center}
\begin{tikzpicture}[x={(240:0.8cm)}, y={(-10:1cm)}, z={(0,1cm)}, 
                   plane max z=3]
    \draw[->] (0,0,0) -- (8,0,0) node[anchor=west]{$u_x$};
    \draw[->] (0,0,0) -- (0,8,0) node[anchor=west]{$u_y$};
    \draw[->] (0,0,0) -- (0,0,4) node[anchor=west]{$u_z$};
        \draw[thick] (4,0,1) -- (0,4,0);
	\draw[thick, gray] (7,0,1) -- (0,7,0);
	\draw[thick, gray] (3,4,0) -- (0,7,0);
	\draw[thick, gray] (7,0,1) -- (4,3,1);
	\draw[dashed] (2,2,3) -- (2,2,3.5);
	\draw[dashed, gray] (4,0,0) -- (0,4,0);
	\draw[dashed, gray] (7,0,0) -- (0,7,0);
	\draw[dashed, gray] (4,0,3) -- (4,0,0);
	\draw[dashed, gray] (0,4,2) -- (0,4,0);
	\draw[dashed, gray] (4,3,0) -- (4,3,1);
	\draw[thick] (4,0,2) -- (0,4,1);
	\draw[thick] (4,0,3) -- (0,4,2);
	\draw[dashed, gray] (7,0,2) -- (7,0,0);
	\draw[thick, gray] (7,0,2) -- (0,7,1);
	\draw[thick, gray] (7,0,2) -- (4,3,2);
	\draw[thick, gray] (3,4,1) -- (0,7,1);
	\draw[dashed, gray] (4,3,2) -- (4,3,0);
	\draw[dashed, gray] (3,4,1) -- (3,4,0);
	\draw[dashed, gray] (0,7,1) -- (0,7,0);
	\node [below] at (6,6,0) {};
        \node [right] at (0,4,0.5) {{$G_{S(4)}(B)$}};
	\node [right] at (0,7,0.5) {{$G_{S(7)}(B)$}};
	\foreach \Point in{(4,0,1), (0,4,0), (7,0,1), (0,7,0),
          (3,4,0), (4,3,1), (4,0,2), (0,4,1), (4,0,3), (0,4,2),
          (7,0,2), (4,3,2), (0,7,1), (3,4,1)}{ \node at \Point
          {\textbullet};  }
\end{tikzpicture}
\end{center}
\caption{Slice graphs for $I(B) = \langle x^{4}z - y^{4}, x^{7}z - y^{7} \rangle$.}
\label{fig:slices}
\end{figure}
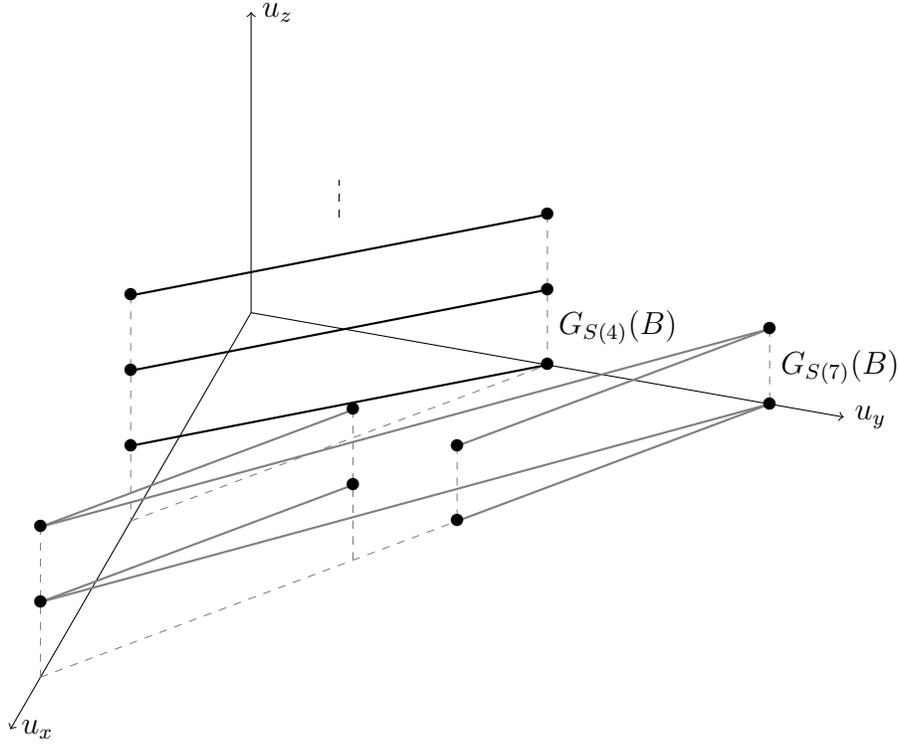

The following result groups the slice graphs of $B$ according
to isomorphism.

\begin{lemma}
\label{prop : translations.} 
  Let $B$ as in Notation~\ref{conv : matrix.}.
  Suppose $(c_x,c_y,c_z) \in \NN^3$ is a vertex of $G_{S(\ell)}(B)$, where
  $\ell \in (1/p)\NN$. 
  Write $c_x = q \bar{c}_x + i$, $c_y = p \bar{c}_y + j$, where $\bar{c}_x,
  \bar{c}_y, i, j \in \NN$ and $0 \leq i < q$, $0 \leq j < p$. Then
  $G_{S(\ell)}(B)$ is isomorphic to the slice graph $G_{S(\ell
    -i-j/\lambda)}(B)$ that contains $(q\bar{c}_x,p\bar{c}_y,c_z)$ as a
  vertex. Consequently, in order to understand the (connected
  components of) all the slice graphs of $B$, it is enough to
  understand $G_{S(\ell)}(B)$ for $\ell \in q\NN$.
\end{lemma}

\begin{proof} 
The desired isomorphism $\varphi_{ij}$ between $G_{S(\ell)}(B)$ and
$G_{S(\ell -i-j/\lambda)}(B)$ is defined by 
\[
(u_x,u_y,u_z) \mapsto
(u_x-i,u_y-j,u_z).
\] 
Note that a vertex of the form $(q\bar{c}_x,p\bar{c}_y,c_z)
\in \NN^3$, where $\bar{c}_x,\bar{c}_y \in\NN$ belongs to a slice graph
$G_{S(\bar{\ell})}(B)$ where $\bar{\ell}\in q\NN$.
\end{proof}

Our next step is to ``straighten out'' the slice graphs of $B$.
Recall the notation of Notation~\ref{conv : matrix.}.
Given $\ell \in \NN$, let $\phi_\ell : \NN^{2} \to \ZZ^{3}$ be the
injective function defined by  
$(w, k) \mapsto (qw, \lambda (q\ell - qw),k)=(qw,p(\ell-w),k)$. 
Note that the image 
$\phi_\ell(\{ (w,k) \in \NN^2 \mid 0\leq w \leq \ell, 0 \leq k \})$ 
is the intersection with $\NN^3$ of the hyperplane
given by $u_y = -\lambda u_x + \lambda q\ell$. 

\begin{lemma}
\label{lemma:straightening graphs}
Let $B$ be as in Notation~\ref{conv : matrix.} and set
$M = \left[ \begin{smallmatrix} r/q & s/q \\ a & b \end{smallmatrix} \right]$.
Given $\ell \in \NN$, define $\phi_\ell$ as above. The image under
$\phi_\ell$ of the band graph $G_\ell(M)$
(see Notation~\ref{notation:bands}) is the slice graph
$G_{S(q\ell)}(B)$. \qed 
\end{lemma}

Since we have already studied the connected components of the band
graphs $G_\ell(M)$, we are ready to compute the primary component of
$I(B)$ associated to $\<x,y\>$.

\begin{theorem}
\label{thm:component in three variables}
Let $B$ as in Notation~\ref{conv : matrix.}. The primary component of
$I(B)$ corresponding to the associated prime $\<x,y\>$ is
\[
(I(B):z^\infty)+
\< y^{\lambda(r + s - d)}, x^{d}y^{\lambda(r + s - 2d)},
x^{2d}y^{\lambda(r + s - 3d)},\dots,x^{r + s -d} \>.
\]
\end{theorem}

\begin{proof}
By Theorem~\ref{thm:component for monomial associated prime}, the
desired component is $(I(B):z^\infty) + 
\mathscr{M}$, where 
\[
\mathscr{M} = \< x^{u_x}y^{u_y} \mid \exists u_z
\in \ZZ \text{ with } u=(u_x,u_y,u_z) \text{ an infinite vertex of }
\sG_{\NN^2\times \ZZ}(I(B))\> . 
\]
By Proposition~\ref{prop : inf ver of
  I and inf ver of J.}, we can use  
$\sG(I(B))$ instead of 
$\sG_{\NN^2\times \ZZ}(I(B))$ in the definition of $\mathscr{M}$, and
by Lemma~\ref{lemma:chain}, we can use $G(B)$ instead of $\sG(I(B))$.

Now combine Remark~\ref{rmk: graphs in hyperplanes}, Lemmas~\ref{prop : translations.}
and~\ref{lemma:straightening graphs} and Theorem~\ref{prop :
  dilation.} to obtain the desired result. 
\end{proof}

\section{Lattice basis ideals and hypergeometric systems}
\label{sec:hypergeometric}

In this section we state and prove our main result,
Theorem~\ref{thm:component}, that gives an explicit expression for the
Andean primary components of a codimension two lattice basis ideal. We
also provide an application to the study of hypergeometric systems of
differential equations.

\begin{theorem}
\label{thm:component}
Let $B$ as in Convention~\ref{conv:AB}, and suppose that $b_i$, $b_j$
are linearly dependent rows of $B$ lying in opposite open quadrants of
$\ZZ^2$. Without loss of generality, assume that $b_{i1},b_{i2} >0$,
let $d = \gcd(b_{i1},b_{i2})$, and write $\lambda = -b_{j1}/b_{i1} = -
b_{j2}/b_{i2} > 0$.
The primary component of
$I(B)$ corresponding to the associated prime $\<x_i,x_j\>$ is
\[
(I(B):(\prod_{k \neq i,j} x_k)^\infty)+
\< x_j^{\lambda(b_{i1} + b_{i2} - d)}, x_i^{d}x_j^{\lambda(b_{i1} + b_{i2}  - 2d)},
x_i^{2d}x_j^{\lambda(b_{i1} + b_{i2}  - 3d)},\dots, x_i^{b_{i1} + b_{i2}  -d} \>.
\]
The only monomials in this ideal are those in
$\< x_j^{\lambda(b_{i1} + b_{i2} - d)}, x_i^{d}x_j^{\lambda(b_{i1} +
  b_{i2}  - 2d)},\dots, x_i^{b_{i1} + b_{i2}  -d} \>$. 
\end{theorem}

\begin{proof}

Let $\sigma = \{i, j\}$, and set $P =\NN^{\sigma} \times \ZZ^{\bar{\sigma}}$.
Choose $0< a \leq b \in \ZZ$ such
that the matrix $\hat{B} = \left[ \begin{smallmatrix} b_{i1} & b_{i2}
    \\ b_{j1} & b_{j2} \\ a & b \end{smallmatrix} \right]$ has rank
$2$.
Our result follows from Theorems~\ref{thm:component in three
  variables} and~\ref{thm:component for monomial associated prime} if we show that 
\begin{align*}
\{ (u_i,u_j) \in \NN^{\sigma} \mid \exists u \in P \text{ an infinite vertex of }
\sG_{P}(I(B)) \}
= \hspace{4cm}   \\ 
\{ (c_1,c_2) \in \NN^2 \mid \exists c_3 \in \ZZ \text{ with }
(c_1,c_2,c_3) \text{ an infinite vertex of } \sG_{\NN^2\times
  \ZZ}(I(\hat{B})).
\end{align*}

By Lemma~\ref{lemma:chain}, it is enough to show that

\begin{align*}
\{ (u_i,u_j) \in \NN^{\sigma} \mid \exists u \in P \text{ with $u_{i}$ and $u_{j}$ an infinite vertex of }
G_{P}(B) \}
= \hspace{2cm}   \\ \numberthis \label{eqn:reduce vars}
\{ (c_1,c_2) \in \NN^2 \mid \exists c_3 \in \ZZ \text{ with }
(c_1,c_2,c_3) \text{ an infinite vertex of } G_{\NN^2\times
  \ZZ}(\hat{B}).
\end{align*}

We show $\subseteq$; the other inclusion is similar.

Let $(u_i,u_j) \in \NN^{\sigma}$, such that there is $u \in \NN^{\sigma}
\times \ZZ^{\bar{\sigma}}$ (whose $i$th and $j$th coordinate are $u_i$
and $u_j$) that is an infinite vertex of $G_{P}(B)$. By
Lemma~\ref{lemma : inf comp.}, 
there are $v,\tv \in P$ connected to $u$ such that $\tv-v \in P$.
Then there is a sequence of vertices $u =
\mu^{(1)},\dots,\mu^{(\ell_1)} = v, \mu^{(\ell_1+1)},\dots,\mu^{(\ell_2)}=\tv \in P$ such that
$(\mu^{(k)},\mu^{(k+1)})$ is an edge in $G_P(B)$ for $k=1,\dots \ell_2-1$. 

Recall that $B_1$ and $B_2$ are the columns of $B$, and denote
$\hat{B}_1$ and $\hat{B}_2$ the columns of $\hat{B}$. Define
\[
v^{(k)} = \left\{
\begin{array}{ll}
\phantom{-}\hat{B}_1 \hspace{1cm} & \text{ if } \mu^{(k+1)}-\mu^{(k)} = B_1 ,\\
-\hat{B}_1 & \text{ if } \mu^{(k+1)}-\mu^{(k)} = -B_1, \\
\phantom{-}\hat{B}_2 \hspace{1cm} & \text{ if } \mu^{(k+1)}-\mu^{(k)} = B_2 ,\\
-\hat{B}_2 & \text{ if } \mu^{(k+1)}-\mu^{(k)} = -B_2.
\end{array}
\right.
\]

Choose any $c \in \ZZ$ and let $\nu^{(1)} = (u_i,u_j,c)$. Set also
$\nu^{(k+1)} = \nu^{(k)} + v^{(k)}$ for $k = 1,\dots, \ell_2-1$. Then
the first and second coordinates of $\nu^{(k)}$ are equal to the $i$th
and $j$th coordinates of $\mu^{(k)}$ respectively, which implies that
$\nu^{(1)},\dots, \nu^{(\ell_2)} \in \NN^2\times \ZZ$. By construction,
$(\nu^{(k)},\nu^{(k+1)})$ is an edge of $G_{\NN^2\times \ZZ}(\hat{B})$ for
$k=1,\dots, \ell_2-1$, so that in particular, $\nu^{(1)}$,
$\nu^{(\ell_1)}$ and $\nu^{(\ell_2)}$ belong to the same connected
component of $G_{\NN^2\times \ZZ}(\hat{B})$.
Moreover,
$\tv-v\in P$ implies that $\nu^{(\ell_2)}-\nu^{(\ell_1)} \in \NN^2
\times \ZZ$, so by Lemma~\ref{lemma : inf comp.}, 
$\nu^{(1)}$ is an infinite vertex of $G_{\NN^2\times \ZZ}(\hat{B})$, and we conclude
that $(u_i,u_j)$ belongs to the right hand side of~\eqref{eqn:reduce vars}. 
\end{proof}

We use Theorem~\ref{thm:component} to study hypergeometric
differential equations. For the remainder of this article, we work
over the field $\kk=\CC$. 

\begin{definition}
\label{def:Horn}
Let $B=(b_{rs})\in \ZZ^{n\times 2}$ as in Convention~\ref{conv:AB},
and let $\kappa \in \CC^n$. We 
use this information to define a system of partial differential
equations in two variables, called a \emph{bivariate hypergeometric system of
  Horn type}, and denoted $\Horn(B,\kappa)$, as follows. (Colored
symbols are added for convenience.) 

\begin{align*}
  \prod_{\mathbin{\color{Red}{\pmb{b_{r1}>0}}}}
  \prod_{\ell=0}^{\mathbin{\color{Red}{\pmb{b_{r1}-1}}}} 
\left(
b_{r1} z_1 \frac{\del \varphi(z_1,z_2)}{\del z_1} + b_{r2} z_2 \frac{\del
    \varphi(z_1,z_2)}{\del z_2} + \kappa_r -\ell \right)  
= \hspace{3cm}\\
\mathbin{\color{Red}{\pmb{z_1}}}
 \prod_{\mathbin{\color{Red}{\pmb{b_{r1}<0}}}}
 \prod_{\ell=0}^{\mathbin{\color{Red}{\pmb{-b_{r1}-1}}}}
  \left( b_{r1} z_1
  \frac{\del \varphi(z_1,z_2)}{\del z_1} + b_{r2} z_2 \frac{\del
    \varphi(z_1,z_2)}{\del z_2} + \kappa_r - \ell \right)  ,\\
  \prod_{\mathbin{\color{Red}{\pmb{b_{r2}>0}}}}
  \prod_{\ell=0}^{\mathbin{\color{Red}{\pmb{b_{r2}-1}}}} 
    \left(
    b_{r1} z_1 \frac{\del \varphi(z_1,z_2)}{\del z_1} + b_{r2} z_2 \frac{\del
    \varphi(z_1,z_2)}{\del z_2} + \kappa_r -\ell \right)  
= \hspace{3cm}\\
\mathbin{\color{Red}{\pmb{z_2}}} 
  \prod_{\mathbin{\color{Red}{\pmb{b_{r2}<0}}}}
  \prod_{\ell=0}^{\mathbin{\color{Red}{\pmb{-b_{r2}-1}}}} 
  \left( b_{r1} z_1
  \frac{\del \varphi(z_1,z_2)}{\del z_1} + b_{r2} z_2 \frac{\del
    \varphi(z_1,z_2)}{\del z_2} + \kappa_r - \ell \right)  .
\end{align*}
\end{definition}

\begin{example}
\label{ex:Appell}
Hypergeometric systems contain as examples many widely studied systems
of differential equations. For example, the Appell system $F_1$:
\begin{align*}
\frac{\del \varphi(z_1,z_2)}{\del z_1}
\left( z_1\frac{\del \varphi(z_1,z_2)}{\del z_1} + 
       z_2\frac{\del \varphi(z_1,z_2)}{\del z_2} +
       \gamma -1 
\right) 
=  \hspace{5cm} \\
\left(
z_1\frac{\del \varphi(z_1,z_2)}{\del z_1} + \beta
\right)
\left( 
z_1\frac{\del \varphi(z_1,z_2)}{\del z_1} + 
z_2 \frac{\del \varphi(z_1,z_2)}{\del z_2} + \alpha
\right) ,
\\
\frac{\del \varphi(z_1,z_2)}{\del z_2}
\left( z_1\frac{\del \varphi(z_1,z_2)}{\del z_1} + 
       z_2\frac{\del \varphi(z_1,z_2)}{\del z_2} +
       \gamma -1 
\right) 
=  \hspace{5cm} \\
\left(
z_1\frac{\del \varphi(z_1,z_2)}{\del z_1} + \beta'
\right)
\left( 
z_1\frac{\del \varphi(z_1,z_2)}{\del z_1} + 
z_2 \frac{\del \varphi(z_1,z_2)}{\del z_2} + \alpha
\right),
\end{align*}
arises from the matrix 
$B = \left[ \begin{smallmatrix} \phantom{-} 1 & \phantom{-} 1 \\
-1 & -1 \\ \phantom{-} 1 & \phantom{-} 0 \\ 
\phantom{-} 0 & \phantom{-} 1 \\ -1 & \phantom{-} 0 \\ \phantom{-} 0 & -1 \end{smallmatrix}\right]$ and the
vector $\kappa = (\gamma - 1, -\alpha, 0,0,-\beta,-\beta') \in \CC^6$.
It is clear that if $\alpha = \gamma -1$, then the space of germs of
holomorphic solutions of the Appell system
near a generic nonsingular point is infinite-dimensional, as any solution of 
$z_1 ({\del \varphi(z_1,z_2)}/{\del z_1}) + 
z_2  ({\del \varphi(z_1,z_2)}/{\del z_2}) + \alpha = 0$ is a solution
of the Appell system. 
The converse of this statement, though non-obvious, is also true.
In fact, the necessary and sufficient constraints on the parameters $\kappa$ for the
system $\Horn(B,\kappa)$ to have a finite dimensional solution space
can be determined using the primary decomposition of the lattice basis
ideal $I(B)$ (see~\cite[Theorem~6.2]{DMM}). 
The relevant object is the
\emph{Andean arrangement} of the lattice basis ideal $I(B)$
(see Definition~\ref{def:qdeg - Andean} below).
\end{example}

Recall from Convention~\ref{conv:AB} that we have been given a matrix
$A \in \ZZ^{(n-2)\times n}$ such that $AB = 0$, and whose columns
$a_1,\dots,a_n$ span $\ZZ^{n-2}$ as a lattice. The ideal $I(B)$ is
homogeneous with respect to the $\ZZ^{n-2}$-grading on $\CC[x]$
defined by $\deg(x_i)=a_i$ for $i=1,\dots, n$. If $V$ is a
$\ZZ^{n-2}$-graded $\CC[x]$-module, we define the set of
\emph{degrees} of $V$ as follows:
\[
\deg(V) = \{ \alpha \in \ZZ^{n-2} \mid V_{\alpha}\neq 0 \}.
\]
\begin{definition}
\label{def:qdeg - Andean}
Let $V$ be a $\ZZ^{n-2}$-graded module and 
consider $\deg(V)$ as a subset of $\CC^{n-2}$. The Zariski closure of
$\deg(V)$ in $\CC^{n-2}$ is called the set of \emph{quasidegrees} of
$V$, denoted $\qdeg(V)$.

If $A$ and $B$ are as in Convention~\ref{conv:AB}, the \emph{Andean
  arrangement} of $I(B)$ is the union over the Andean components
$\mathscr{C}$ of $I(B)$ of $\qdeg(\CC[x]/\mathscr{C})$.
\end{definition}

The Andean arrangement of $I(B)$ is a finite union of translates of
some faces of the cone $\RR_>0 A$ of non-negative combinations of the
columns of $A$ (see \cite[Lemma~6.2]{DMM}), so it is, indeed, an
arrangement of affine spaces.

The importance of this notion can be seen in the following result,
which is a special case of \cite[Theorem~6.3]{DMM}.

\begin{theorem}
\label{thm:finite rank}
Let $A$ and $B$ as in Convention~\ref{conv:AB}. Assume that
the span of the columns of $B$ contains no nonzero coordinatewise non
negative element (equivalently, the cone over the columns of $A$
contains no lines). 
The Horn system $\Horn(B,\kappa)$ has a finite dimensional solution
space if and only if $A\kappa$ lies outside the Andean arrangement of
$I(B)$. 
\end{theorem}

In fact, we can state a stronger result. A left ideal $J$ in the ring $D_\ell$ of
linear partial differential operators with polynomial coefficients in $\ell$
variables is \emph{holonomic} if $\Ext^j(D_\ell/J,D_\ell) = 0$ whenever $j\neq \ell$.
Holonomic ideals have finite dimensional solution spaces; the converse
of this statement, while untrue in general (see \cite[Example~9.1]{BMW}), does
hold for bivariate Horn systems.

\begin{theorem}[Corollary~9.6 in~\cite{BMW}]
\label{thm:holonomic}
Let $A$ and $B$ be as in Convention~\ref{conv:AB}.
Assume that
the span of the columns of $B$ contains no nonzero coordinatewise non
negative element (equivalently, the cone over the columns of $A$
contains no lines). 
The Horn system $\Horn(B,\kappa)$ is holonomic if and only if
$A\kappa$ lies outside the Andean arrangement of $I(B)$. 
\end{theorem}

As a consequence of Theorem~\ref{thm:component}, we are able to
explicitly compute the Andean arrangement of $I(B)$ for a codimension
two lattice basis ideal, thus providing a concrete description of the
set of parameters for which a Horn system in two variables is
holonomic (and has finite dimensional solution space). We first
compute the quasidegree set for a single Andean component of $I(B)$.

\begin{proposition}
\label{prop:qdeg one component}
Let $B$ be as in Convention~\ref{conv:AB}, and suppose
$b_i,b_j,\lambda, d$ are as in Theorem~\ref{thm:component}. 
Denote by $\mathscr{C}$ the component of $I(B)$ associated to
$\<x_i,x_j\>$. Write
$\lambda = p/q$ where $p$ and $q$ are relatively prime integers, and
assume that the vector with $i$th coordinate $p$, $j$th coordinate $q$
and all other coordinates equal to zero is the first row of $A$. Then
\[
\scriptsize{
\qdeg\left( \frac{\CC[x]}{\mathscr{C}} \right) = 
\left\{
\begin{array}{ll}
\big\{ \beta \in \CC^{n-2} \mid \beta_1 \in  \{0,1,\dots,
pb_{i1}+p b_{i2} -1\} \minus \{ p b_{i1}+p b_{i2}-p \} \; \big\}
& \text{if } d/q=1, \\ 
\big\{ \beta \in \CC^{n-2} \mid \beta_1 \in \{0,1,\dots, p b_{i1}+p b_{i2}-1
\} \; \big\}& \text{if } d/q>1.
\end{array}
\right.
}
\]
\end{proposition}

\begin{proof}
We start by justifying the assumption that the first row of $A$
consists of zeros except for the $i$th coordinate which equals $p$ and
the $j$th coordinate which equals $q$. We have a choice
of $A$ in Convention~\ref{conv:AB}: any $A$ whose rows are a basis of the left
kernel of $B$, and whose columns span $\ZZ^d$ as a lattice can be
used. Since the row we want is an element of the left kernel of $B$,
and its only two nonzero entries are relatively prime, we are allowed
to use it.

As $b_i$ and $b_j$ are linearly dependent, the
submatrix of $A$ with columns $a_k$ for $k\neq i,j$ has rank
$n-3$. Because the first row of this submatrix consists of all zeros,
we see that, fixing natural numbers $u_i,u_j$ and varying $u_k \in
\NN$ for $k\neq i, j$, the points $Au$ are Zariski dense in the set 
$\{ \beta \in \CC^{n-3} \mid \beta_1 = pu_1+qu_2\}$. 

Now note that the degree $\alpha$ graded piece
$(\CC[x]/\mathscr{C})_{\alpha}$ equals $0$ for $\alpha \in 
\ZZ^{n-2}$ if and only if $\CC[x]_\alpha \subseteq \mathscr{C}$; also
$\CC[x]_\alpha$ is spanned by monomials, and the only monomials in
$\mathscr{C}$ are those in
$\mathscr{M} = \< x_j^{\lambda(b_{i1} + b_{i2} - d)}, x_i^{d}x_j^{\lambda(b_{i1} +
  b_{i2}  - 2d)},\dots, x_i^{b_{i1} + b_{i2}  -d} \>$. 

Thus,
\[
\qdeg(\CC[x]/\mathscr{C}) = 
\big\{ \beta \in
\CC^{n-3} \mid \beta_1 \in \{ pu_i+qu_j \mid u_i, u_j \in \NN \text{
  and } x_i^{u_i}x_j^{u_j}
\not \in \mathscr{M} \} \big\} .
\]

For fixed $\ell$, consider the line $L_{\ell}  = \{ (w,z) \mid
pu_i+qu_j =  \ell\}$. Given $\ell \in \NN$;
we wish to know wether or not all $(u_i,u_j) \in \NN^2 \cap L_\ell$
are exponent vectors of monomials in $\mathscr{M}$. By construction of
$\mathscr{M}$ this occurs if $ \ell \geq pb_{i1} + qb_{i2}$. If
$\ell < pb_{i1} + qb_{i2}$, then $L_\ell$ always contains an exponent vector
of a monomial \emph{not} in $\mathscr{M}$, unless $d=q$, in which case
the intersection $\NN^2 \cap L_{p(b_{i1}+b_{i2}-1)}$ consists of
exponent vectors of monomials in $\mathscr{M}$.
\end{proof}

\begin{theorem}
\label{thm:Andean arrangement}
Let $B$ be as in Convention~\ref{conv:AB}, and let 
\[
\mathscr{S} = \{ \{i,j\} \mid \text{ the rows } b_i \text{ and } b_j \text{
  of } B \text{ are linearly dependent in opposite open quadrants }\}.
\]
For each $\sigma = \{i,j\} \in \mathscr{S}$, assume $i<j$ and let
$\lambda_{\sigma} = -b_{j1}/b_{i1} = -b_{j2}/b_{i2}$. Write
$ \lambda_\sigma = p_\sigma/q_\sigma$ for $p_\sigma, q_\sigma \in \NN$
relatively prime. We can find $h_\sigma \in \QQ^{n-2}$ such that 
$h_\sigma A$ has $i$th coordinate $p_\sigma$, $j$th coordinate
$q_\sigma$, and all other entries equal to zero.
The Andean arrangement of $I(B)$ equals
\[\scriptsize{
\bigcup_{\sigma =\{i,j\}\in \mathscr{S}} 
\left\{ 
\beta \in \CC^{n-2} \left| 
h_\sigma \cdot \beta \in \left\{
\begin{array}{ll}
\{0,1,\dots, p_\sigma |b_{i1}|+p_\sigma |b_{i2}|-1 \} \minus \{
p_\sigma |b_{i1}|+p_\sigma |b_{i2}| - p_\sigma\} &
\text{if } d_\sigma/q_\sigma=1,\\
\{ 0,1,\dots, p_\sigma |b_{i1}|+p_\sigma |b_{i2}|-1 \} & \text{if } d_\sigma/q_\sigma>1.
\end{array}
\right.
\right.
\right\}.
}
\]
\end{theorem}

\begin{example}
\label{ex:Appell continued}
We apply the previous result to the case Appell $F_1$, and compute the
Andean arrangement of $I(B)$ for 
$B = \left[ \begin{smallmatrix} 
\phantom{-} 1 & \phantom{-} 1 \\
-1 & -1 \\ 
\phantom{-} 1 & \phantom{-} 0 \\ 
\phantom{-} 0 & \phantom{-} 1 \\ 
-1 & \phantom{-} 0 \\ 
\phantom{-} 0 &-1 \end{smallmatrix}\right]$. 
Only the first two rows of $B$ are
linearly dependent in opposite open quadrants of $\ZZ^2$, and we can
choose 
$A = \left[ \begin{smallmatrix} 
1 & 1 & 0 & 0 & 0 & 0 \\
0 & 0 & 1 & 0 & 1 & 0 \\ 
0 & 0 & 0 & 1 & 0 & 1 \\ 
1 & 0 & 0 & 0 & 1 &1 \end{smallmatrix}\right]$. The Andean arrangement
of $I(B)$ equals 
\[
\{ \beta \in \CC^4 \mid \beta_1 \in \{0, 1\} \minus \{1\} \} =\{ \beta
\in \CC^4 \mid \beta_1=0 \}.
\]
By Theorems~\ref{thm:finite rank} and~\ref{thm:holonomic}, the system
$\Horn(B,\kappa)$ for $\kappa = (\gamma-1,\alpha,0,0,-\beta,-\beta')$
is holonomic if and only if it has a finite dimensional solution
space, if and only if $\beta_1 = \gamma-1 - \alpha = 0$, as we claimed
in Example~\ref{ex:Appell}.
\end{example}

At first glance, there is no obvious relationship between the Horn system
$\Horn(B,\kappa)$ and the lattice basis ideal $I(B)$, so it is unclear
why the Andean arrangement of $I(B)$ influences the behavior of
$\Horn(B,\kappa)$. The key idea is due to Gelfand, Graev, Kapranov and
Zelevinsky, who introduced a system
\[
\begin{split}
\frac{\del^{|u_+|} \psi(y_1,\dots,y_n)}{\del y_1^{(u_+)_1}\cdots
    \del y_n^{(u_+)_n}}  = 
\frac{\del^{|u_-|} \psi(y_1,\dots,y_n)}{\del y_1^{(u_-)_1}\cdots
    \del y_n^{(u_-)_n}} \hspace{1.1cm} & \text{ for } u \text{ a column
    of } B, \\
 \sum_{j=1}^n a_{ij} \frac{\psi(y_1,\dots,y_n)}{\del y_j}   = (A\kappa)_i
 \psi(y_1,\dots,y_n)  \hspace{1cm} & \text{ for } i=1,\dots n-2,
\end{split}
\]
denoted $H(B,A\kappa)$, 
and noticed that if $\varphi(z_1,z_2)$ is a solution of
$\Horn(B,\kappa)$ then $y^\kappa\varphi(y^{B_1},y^{B_2})$ is a
solution of $H(B,A\kappa)$ (recall that $B_1$ and $B_2$ are the
columns of $B$). Since $H(B,A\kappa)$ contains the generators of the
lattice basis ideal $I(B)$ as differential equations, it is
natural that the behavior of $I(B)$ should impact $H(B,A\kappa)$.
We illustrate this mechanism in an example.

\begin{example}
\label{ex: non holonomic}
Let $B=\left[ \begin{smallmatrix} \phantom{-}2 & \phantom{-}4 \\ -2 &
    -4 \\ \phantom{-}1 & \phantom{-}1 \end{smallmatrix} \right]$. We
use $A=\left[ \begin{smallmatrix} 1 & 1 & 0 \end{smallmatrix}\right]$.
The
Andean arrangement of $I(B)$ equals the quasidegree set of the
component corresponding to the associated prime $\<x_1,x_2\>$. In this
case, $\lambda=1=p=q$, and $\gcd(2/q,4/q)=2$, so we know that $A\kappa
= 2+4-2=4$ implies that $\Horn(B,\kappa)$ is not holonomic, and has
infinite dimensional solution space. Let us exhibit infinitely many
linearly independent solutions for $\Horn(B,(1,3,0))$.

It is easier to produce solutions if we start by considering
$H(B,4)$. Intuitively, the vertices in a connected component of
$\sG(I(B))$ that lies in the $\ell$th slice graph of $I(B)$, 
are the monomials appearing in a series expansion of a
solution of $H(B,A \cdot (\lambda \ell, 0,0) ) = H(B,\ell)$. Since
their monomial sets are disjoint, solutions arising from different
connected components are linearly independent.
Each slice graph contains infinitely many vertices, so the only way
this method produces a finite dimensional solution space is if the
connected components contain infinitely many vertices (and thus give
rise to solutions that are infinite series). If the connected
components of a given slice graph are all finite, there must be
infinitely many of them, and this provides an infinite dimensional
subspace of the solution space of $H(B,\ell)$.

More concretely, for $\ell = 4$ and any $k \in \NN$, $\{ (1,3,k),
(3,1,k+1)\}$ is a connected component of $\sG_{S(4)}(I(B))$. We
propose a solution of $H(B,4)$ which is a linear combination of $y_1
y_2^3 y_3^k$ and $y_1^3 y_2 y_3^{k+1}$. Solving for the coefficients,
we see that $f_k(y_1,y_2,y_3) = (k+1) y_1 y_2^3 y_3^k + y_1^3 y_2
y_3^{k+1}$ is a solution of $H(B,4)$.

We know that if we can write $f(y_1,y_2,y_3) = y_1 y_2^3 F(y_1^2y_3/y_2^2,
y_1^4 y_3/y_2^4)$, and $f(y_1,y_2,y_3)$ is a solution of $H(B,4)$,
then $F(z_1,z_2)$ is a solution of $\Horn(B,(1,3,0))$. Using the
expression $y_3 =
(y_1^2y_3/y_2^2)^2/(y_1^4y_3/y_2^4)$, we see that  $f_k(y_1,y_2,y_3) =
y_1 y_2^3 F_k(y_1^2y_3/y_2^2, y_1^4 y_3/y_2^4)$, where $F_k(z_1,z_2) =
(k+1)(z_1^2/z_2)^k(1+z_1)$. The functions $F_k(z_1,z_2)$ span an
infinite dimensional subspace of the solution space of
$\Horn(B,(1,3,0))$.
Note that in this example, there is no obvious common factor in the
equations $\Horn(B,(1,3,0))$, as was the case for the Appell system $F_1$.
\end{example}

\end{document}